\documentclass[11pt,a4paper]{article}
\usepackage[T1]{fontenc}
\usepackage[utf8]{inputenc}
\usepackage{lmodern}
\usepackage{amsmath,amssymb,amsthm,mathtools}
\usepackage[margin=26mm,headheight=14pt]{geometry}
\usepackage{microtype}
\usepackage{xcolor}
\usepackage{fancyhdr}
\usepackage[colorlinks=true,linkcolor=blue!45!black,citecolor=blue!45!black,urlcolor=blue!45!black]{hyperref}
\hypersetup{pdftitle={A Gaussian Chain-Rule Proof of the Banach-Space Hanson-Wright Bound},pdfsubject={The matrix-family estimate and Conjecture 2 of Adamczak, Latala and Meller}}
\allowdisplaybreaks[1]
\numberwithin{equation}{section}
\newtheorem{theorem}{Theorem}[section]
\newtheorem{lemma}[theorem]{Lemma}
\newtheorem{proposition}[theorem]{Proposition}

\theoremstyle{remark}
\newtheorem{remark}[theorem]{Remark}
\newcommand{\E}{\mathbb E}
\newcommand{\PP}{\mathbb P}
\newcommand{\R}{\mathbb R}
\newcommand{\cA}{\mathcal A}
\newcommand{\cF}{\mathcal F}
\newcommand{\norm}[1]{\lVert #1\rVert}
\newcommand{\abs}[1]{\lvert #1\rvert}
\newcommand{\ip}[2]{\langle #1,#2\rangle}
\newcommand{\HS}{\mathrm{HS}}
\newcommand{\op}{\mathrm{op}}
\newcommand{\rad}{\operatorname{rad}}
\newcommand{\diam}{\operatorname{diam}}
\newcommand{\Lip}{\operatorname{Lip}}

\title{\vspace{-7mm}A Gaussian Chain-Rule Proof of the\\Banach-Space Hanson--Wright Bound}
\author{{Witold Bednorz, Rafal Martynek and Rafal Meller}
\footnote{Research partially supported by  Grant UMO-2022/47/B/ST1/02114}
\footnote{Institute of Mathematics, University of Warsaw, Banacha 2, 02-097 Warszawa, Poland}}
\date{}

\begin{document}
\maketitle
\vspace{-5mm}
\begin{abstract}
For a finite family of real matrices, we bound the Gaussian width of the union of its image ellipsoids by its maximal Hilbert--Schmidt norm, its maximal fixed-input Gaussian width, and the square root of the product of its operator radius and decoupled Gaussian chaos supremum. The proof applies the linear case of Maurer's Gaussian chain rule to the conditional image set $\{\pm A^Tg:A\in\cA\}$, and then uses Gaussian symmetrization. This gives the estimate in Conjecture 24 of Adamczak, Lata{\l}a and Meller and proves their Conjecture 2. We include the passage to Banach-space-valued quadratic forms and a direct chaining proof of the linear comparison.
\end{abstract}

\section{The matrix-family estimate}

Let $\cA$ be a nonempty finite family of real $m\times d$ matrices. Gaussian vectors are standard; distinct Gaussian vectors appearing together are independent unless explicitly related by a rotation. Define
\begin{align}
 W&=\E_g\sup_{A\in\cA}\norm{A^Tg}_2,
 &S&=\E_{g,h}\sup_{A\in\cA}\abs{g^TAh},\label{eq:WS}\\
 B&=\sup_{\norm{x}_2\le1}\E_g\sup_{A\in\cA}\abs{g^TAx},
 &H&=\sup_{A\in\cA}\norm{A}_{\HS},
 \qquad D=\sup_{A\in\cA}\norm{A}_{\op}.
 \label{eq:BHD}
\end{align}
Here $g\in\R^m$ and $h,x\in\R^d$. The letter $C$ denotes a numerical constant whose value may change between occurrences. The notation $a\lesssim b$ means $a\le Cb$.

\begin{theorem}\label{thm:main}
For every such family,
\begin{equation}
 \boxed{\quad W\lesssim H+B+\sqrt{DS}.\quad}
 \label{eq:main}
\end{equation}
Consequently, for every $p\ge1$,
\begin{equation}
 W\lesssim \frac{S}{\sqrt p}+B+H+\sqrt p\,D.
 \label{eq:pbound}
\end{equation}
The constants are independent of the dimensions and the cardinality of $\cA$.
\end{theorem}

The quantity $W$ is the Gaussian width of $\bigcup_{A\in\cA}AB_2^d$. The parameters in \eqref{eq:pbound} agree with those in equations (27)--(28) of \cite{ALM}. No symmetry or diagonal restriction on the matrices is used in Theorem~\ref{thm:main}.

\section{A linear consequence of the Gaussian chain rule}

For a bounded set $T\subset\R^d$, write
\[
 w(T)=\E_h\sup_{x\in T}\ip{h}{x},
 \qquad \rad(T)=\sup_{x\in T}\norm{x}_2.
\]
For symmetric $T$, the supremum defining $w(T)$ can equivalently be taken with an absolute value.

\begin{lemma}[Linear Gaussian chain rule]\label{lem:chain}
Let $T\subset\R^d$ be finite, symmetric, and contain $0$. With $B,D$ defined by \eqref{eq:BHD},
\begin{equation}
 \E_z\sup_{\substack{A\in\cA\\x\in T}}\abs{\ip{z}{Ax}}
 \lesssim D\,w(T)+B\,\rad(T).
 \label{eq:linear-chain}
\end{equation}
\end{lemma}

\begin{proof}
We apply Theorem 2 of Maurer~\cite{Maurer} to
\[
 \cF=\{x\mapsto \sigma Ax:A\in\cA,\ \sigma\in\{-1,1\}\}.
\]
That theorem states, with universal constants, that
\begin{equation}
 w(\cF(T))\lesssim
 \Lip(\cF)w(T)+\diam(T)\mathcal R(\cF,T)+w(\cF(0)),
 \label{eq:maurer}
\end{equation}
where
\[
 \mathcal R(\cF,T)=
 \sup_{\substack{x,y\in T\\x\ne y}}
 \frac{\E_z\sup_{f\in\cF}\ip{z}{f(x)-f(y)}}{\norm{x-y}_2}.
\]
The supremum over an empty collection of pairs is taken to be zero. In our linear setting,
\[
 \Lip(\cF)\le D,
 \qquad
 \mathcal R(\cF,T)
 =\sup_{x\ne y}\frac{\E_z\sup_A\abs{z^TA(x-y)}}{\norm{x-y}_2}
 \le B.
\]
Moreover, $\diam(T)\le2\rad(T)$ and $\cF(0)=\{0\}$. Substitution into \eqref{eq:maurer} proves the lemma. A direct proof of the linear case is given in Appendix~\ref{app:chain}.
\end{proof}

\section{The correlated quadratic comparison}

Define
\begin{equation}
 Q=\E_{g,z}\sup_{A\in\cA}\abs{z^TAA^Tg},
 \qquad g,z\in\R^m.
 \label{eq:Q}
\end{equation}

\begin{proposition}\label{prop:Q}
For every finite family $\cA$,
\begin{equation}
 \boxed{\quad Q\lesssim DS+BW.\quad}
 \label{eq:Qbound}
\end{equation}
\end{proposition}

\begin{proof}
Fix $g$ and form the actual conditional image set
\[
 T_g=\{0\}\cup\{\pm A^Tg:A\in\cA\}\subset\R^d.
\]
Set
\[
 r_g=\sup_A\norm{A^Tg}_2,
 \qquad s_g=\E_h\sup_A\abs{h^TA^Tg}.
\]
Then
\begin{equation}
 \rad(T_g)=r_g,\qquad w(T_g)=s_g,
 \qquad \E_g r_g=W,\qquad \E_g s_g=S.
 \label{eq:conditional-parameters}
\end{equation}
The vector $z$ is independent of $g$, so Lemma~\ref{lem:chain} applies with $T_g$ held fixed. Enlarging the diagonal choice of two matrices gives
\begin{align*}
 \E_z\sup_A\abs{z^TAA^Tg}
 &\le \E_z\sup_{A,C\in\cA}\abs{z^TAC^Tg}\\
 &=\E_z\sup_{\substack{A\in\cA\\x\in T_g}}\abs{z^TAx}
 \lesssim D s_g+B r_g.
\end{align*}
Averaging over $g$ and using \eqref{eq:conditional-parameters} proves \eqref{eq:Qbound}. All sets are finite, so measurability causes no difficulty.
\end{proof}

\begin{proof}[Proof of Theorem~\ref{thm:main}]
Let $\widetilde g$ be an independent copy of $g$. Since
$\E\norm{A^Tg}_2^2=\norm{A}_{\HS}^2$, Jensen's inequality gives
\begin{align}
 W^2
 &\le \E_g\sup_A\norm{A^Tg}_2^2\notag\\
 &\le H^2+\E_g\sup_A\bigl(\norm{A^Tg}_2^2-\norm{A}_{\HS}^2\bigr)\notag\\
 &\le H^2+\E_{g,\widetilde g}\sup_A
 \bigl(\norm{A^Tg}_2^2-\norm{A^T\widetilde g}_2^2\bigr).
 \label{eq:symmetrization}
\end{align}
For the last step, the supremum is convex in the finite vector of its arguments, and expectation over $\widetilde g$ is moved outside it.

The rotation
\[
 u=\frac{g+\widetilde g}{\sqrt2},\qquad
 v=\frac{g-\widetilde g}{\sqrt2}
\]
produces independent standard Gaussian vectors. For every $A$,
\[
 \norm{A^Tg}_2^2-\norm{A^T\widetilde g}_2^2
 =2\ip{A^Tu}{A^Tv}=2u^TAA^Tv.
\]
Thus \eqref{eq:symmetrization} and Proposition~\ref{prop:Q} imply
\begin{equation}
 W^2\le H^2+2Q\le H^2+C BW+C DS.
 \label{eq:quadratic-closure}
\end{equation}
Absorb $CBW$ using
$CBW\le \tfrac12W^2+\tfrac12C^2B^2$.
This proves \eqref{eq:main}. Finally,
\[
 \sqrt{DS}\le\frac12\left(\frac{S}{\sqrt p}+\sqrt p\,D\right)
\]
gives \eqref{eq:pbound} for every $p\ge1$.
\end{proof}

\begin{remark}[Why the correlations are retained]
The application of the chain rule is conditional on the original vector $g$. The same Gaussian vector $z$ acts on every point $AC^Tg$; no independent multiplier is assigned to a group of matrices. Operator distinctions that give the same point of $T_g$ disappear automatically. The conditional image contributes $Ds_g$, and the Gaussian increment complexity of the linear maps contributes $Br_g$. The latter becomes $BW$ and is absorbed only after symmetrization.
\end{remark}

\begin{remark}[Compact families]
Theorem~\ref{thm:main} extends to compact matrix families by taking increasing finite subsets of a countable dense subset. The suprema converge pointwise; monotone convergence applies after taking the symmetric families and adjoining zero. In finite dimensions every bounded family can first be replaced by its compact closure. The same constants apply.
\end{remark}

\section{Banach-space-valued quadratic forms}

Let $(F,\norm{\cdot}_F)$ be a real normed space and let $a_{ij}=a_{ji}\in F$, $1\le i,j\le n$. Only their finite-dimensional span matters. Define
\begin{align*}
 X(g)&=\sum_{i,j=1}^n a_{ij}(g_i g_j-\delta_{ij}),
 &M&=\E\norm{X(g)}_F,\\
 H_F&=\sup_{\sum_{i,j}x_{ij}^2\le1}
       \left\|\sum_{i,j}a_{ij}x_{ij}\right\|_F,
 &D_F&=\sup_{\norm{x}_2,\norm{y}_2\le1}
       \left\|\sum_{i,j}a_{ij}x_i y_j\right\|_F,\\
 B_0&=\sup_{\norm{x}_2\le1}\E_g
       \left\|\sum_{i\ne j}a_{ij}g_i x_j\right\|_F.
\end{align*}

\begin{theorem}[Conjecture 2 of \cite{ALM}]\label{thm:banach}
For every $p\ge1$,
\begin{equation}
 \boxed{\quad
 \bigl(\E\norm{X(g)}_F^p\bigr)^{1/p}
 \lesssim M+\sqrt p\,B_0+\sqrt p\,H_F+pD_F.
 \quad}
 \label{eq:banach}
\end{equation}
\end{theorem}

We give the reduction explicitly. Introduce the full linear quantities
\begin{align*}
 W_F&=\E_g\sup_{\norm{x}_2\le1}
       \left\|\sum_{i,j}a_{ij}g_i x_j\right\|_F,\\
 B_F&=\sup_{\norm{x}_2\le1}\E_g
       \left\|\sum_{i,j}a_{ij}g_i x_j\right\|_F,
 \qquad
 S_F=\E_{g,h}\left\|\sum_{i,j}a_{ij}g_i h_j\right\|_F.
\end{align*}

\begin{lemma}\label{lem:moment}
For every $p\ge1$,
\begin{equation}
 \bigl(\E\norm{X(g)}_F^p\bigr)^{1/p}
 \lesssim M+\sqrt p\,W_F+pD_F.
 \label{eq:standard-moment}
\end{equation}
\end{lemma}

\begin{proof}
This is the standard Gaussian concentration estimate underlying equation (3) of \cite{ALM}. For completeness, let
\[
 K=\left\{z:\norm{X(z)}_F\le4M,\quad
 \sup_{\norm{x}_2\le1}\left\|\sum_{i,j}a_{ij}z_i x_j\right\|_F\le4W_F\right\}.
\]
Markov's inequality gives $\gamma_n(K)\ge1/2$. If $z\in K$ and $\norm{u}_2\le1$, symmetry of the coefficients gives
\[
 X(z+tu)=X(z)+2t\sum_{i,j}a_{ij}z_i u_j
                    +t^2\sum_{i,j}a_{ij}u_i u_j.
\]
Consequently,
\[
 \norm{X(z+tu)}_F\le4M+8tW_F+t^2D_F.
\]
Gaussian isoperimetry therefore yields, for $t\ge0$,
\[
 \PP\{\norm{X(g)}_F>4M+8tW_F+t^2D_F\}\le e^{-t^2/2}.
\]
Integration of this tail bound proves \eqref{eq:standard-moment}. If $M=0$, the quadratic polynomial vanishes almost surely and the assertion is immediate.
\end{proof}

\begin{proof}[Proof of Theorem~\ref{thm:banach}]
Work in the finite-dimensional span of the coefficients. For $\varphi$ in its dual unit ball, let
\[
 A_\varphi=(\varphi(a_{ij}))_{i,j=1}^n.
\]
The image of the dual unit ball is a compact matrix family. Duality identifies its parameters from \eqref{eq:WS}--\eqref{eq:BHD} with
\[
 W=W_F,\quad S=S_F,\quad B=B_F,\quad H=H_F,\quad D=D_F.
\]
For example,
\[
 \sup_\varphi\norm{A_\varphi}_{\HS}
 =\sup_{\sum_{i,j}x_{ij}^2\le1}
   \sup_\varphi\abs{\varphi\!\left(\sum_{i,j}a_{ij}x_{ij}\right)}=H_F.
\]
Theorem~\ref{thm:main} consequently gives
\begin{equation}
 W_F\lesssim H_F+B_F+\sqrt{D_FS_F}.
 \label{eq:WF}
\end{equation}

Two elementary observations identify the required terms. First, if $g,h$ are independent and
$u=(g+h)/\sqrt2$, $v=(g-h)/\sqrt2$, then
\[
 X(u)-X(v)=2\sum_{i,j}a_{ij}g_i h_j.
\]
Both $u$ and $v$ are standard Gaussian, so the triangle inequality implies
\begin{equation}
 S_F\le M.
 \label{eq:SFM}
\end{equation}
Second, the definition of $H_F$ gives, for every $\norm{x}_2\le1$,
\[
 \E\left\|\sum_i a_{ii}g_i x_i\right\|_F
 \le H_F\E\left(\sum_i g_i^2x_i^2\right)^{1/2}
 \le H_F.
\]
Hence
\begin{equation}
 B_F\le B_0+H_F.
 \label{eq:diagonal}
\end{equation}
Insert \eqref{eq:WF}--\eqref{eq:diagonal} into Lemma~\ref{lem:moment}:
\[
 \bigl(\E\norm{X(g)}_F^p\bigr)^{1/p}
 \lesssim M+\sqrt p\,(B_0+H_F)+\sqrt{pD_FM}+pD_F.
\]
Finally, $\sqrt{pD_FM}\le(M+pD_F)/2$ proves \eqref{eq:banach}.
\end{proof}

Theorem~\ref{thm:main} also proves Conjecture 24 of \cite{ALM} directly: for its parameter $t$, take the matrix with entries $\sum_k a_{ijk}t_k$ and apply \eqref{eq:pbound}. The finite-family case suffices there, and bounded parameter sets follow by closure and approximation.

\clearpage
\appendix
\section{A direct chaining proof of the linear comparison}\label{app:chain}

This appendix proves Lemma~\ref{lem:chain} from Gaussian concentration and the Gaussian majorizing-measure theorem. Its role is to make explicit how the unweighted mean cost and the weighted fluctuation cost can be paid by the same approximations.

For a finite symmetric $T\subset\R^d$ containing $0$, set $R=\rad(T)$. The majorizing-measure theorem supplies finite sets $T_k\subset T$ and maps $\pi_k:T\to T_k$ such that, after changing universal constants and shifting the levels if necessary,
\begin{equation}
 T_0=\{0\},\quad |T_k|\le2^{2^k},\quad
 \Gamma:=\sup_{t\in T}\sum_{k\ge0}2^{k/2}
                 \norm{t-\pi_k(t)}_2\lesssim w(T).
 \label{eq:admissible}
\end{equation}
One may arrange that $T_k=T$ and $\pi_k(t)=t$ at all sufficiently large levels. Replacing the original singleton at level zero by $0$ costs at most $R$, and $w(T)\ge\sqrt{2/\pi}\,R$. This justifies the stated convention. We use only this standard form of the Gaussian majorizing-measure theorem; see also the background in Section 4 of \cite{Maurer}.

Fix $t\in T$ and write $r_k(t)=\norm{t-\pi_k(t)}_2$. Starting with $j_0=0$, retain the next approximation at the first index $j_{\ell+1}>j_\ell$ for which
\[
 r_{j_{\ell+1}}(t)\le\tfrac12 r_{j_\ell}(t).
\]
Stop when the error is zero. Let
\[
 \Delta_{j_{\ell+1}}(t)=\pi_{j_{\ell+1}}(t)-\pi_{j_\ell}(t),
\]
and set $\Delta_k(t)=0$ at all other levels. Then $t=\sum_{k\ge1}\Delta_k(t)$ and the retained errors decrease geometrically. In particular,
\begin{equation}
 \sum_{k\ge1}\norm{\Delta_k(t)}_2\le3\norm{t}_2\le3R.
 \label{eq:length-budget}
\end{equation}
If $k$ is a retained index and $j<k$ is the preceding retained index, then either $k=j+1$, or minimality of $k$ gives $r_{k-1}(t)>r_j(t)/2$. In both cases,
\[
 \norm{\Delta_k(t)}_2\le r_j(t)+r_k(t)\le3r_{k-1}(t).
\]
Consequently,
\begin{equation}
 \sup_{t\in T}\sum_{k\ge1}2^{k/2}\norm{\Delta_k(t)}_2
 \le3\sqrt2\,\Gamma.
 \label{eq:weighted-budget}
\end{equation}

For each Gaussian vector $z\in\R^m$, define the seminorm
\[
 N_z(v)=\sup_{A\in\cA}\abs{z^TAv}.
\]
For fixed $v$, we have $\E N_z(v)\le B\norm{v}_2$, and $z\mapsto N_z(v)$ is $D\norm{v}_2$-Lipschitz. Gaussian concentration therefore gives
\begin{equation}
 \PP\{N_z(v)>B\norm{v}_2+sD\norm{v}_2\}\le e^{-s^2/2}
 \qquad(s\ge0).
 \label{eq:edge-concentration}
\end{equation}
Every nonzero edge at level $k$ belongs to
\[
 \mathcal E_k=\bigcup_{0\le j<k}(T_k-T_j),
 \qquad |\mathcal E_k|\le\exp(C_0 2^k)
\]
for a universal $C_0$. A union bound in \eqref{eq:edge-concentration} gives the following statement for a sufficiently large universal $C_1$: with probability at least $1-Ce^{-u^2/2}$, simultaneously for every $k\ge1$ and $v\in\mathcal E_k$,
\begin{equation}
 N_z(v)\le
 \bigl(B+C_1D2^{k/2}+uD\bigr)\norm{v}_2,
 \qquad u\ge0,
 \label{eq:all-edges}
\end{equation}
Subadditivity of $N_z$, followed by \eqref{eq:length-budget} and \eqref{eq:weighted-budget}, yields on this event
\[
 \sup_{t\in T}N_z(t)
 \le 3BR+C D\Gamma+3uDR.
\]
Integrating in $u$ gives
\[
 \E\sup_{t\in T}N_z(t)\lesssim BR+D\Gamma+DR.
\]
Finally, $B\ge\sqrt{2/\pi}\,D$: choose a matrix and a unit vector approaching the maximal operator norm and use the mean absolute value of a scalar Gaussian. The term $DR$ is therefore absorbed by $BR$, while \eqref{eq:admissible} bounds $\Gamma$ by $w(T)$. This proves \eqref{eq:linear-chain}.

\end{document}